\documentclass[12pt,a4paper]{amsart}
\usepackage[utf8]{inputenc}
\usepackage[T1]{fontenc}
\usepackage{amsmath,amssymb,amsthm}
\usepackage{amsaddr}
\usepackage{xcolor}

\usepackage{placeins}

\usepackage[ocgcolorlinks, linkcolor=red]{hyperref}
\usepackage[nocompress]{cite} 

\usepackage[normalem]{ulem}

\usepackage{mathtools}
\mathtoolsset{showonlyrefs=true}

\newtheorem{definition}{Definition}[section]
\newtheorem{theorem}{Theorem}[section]
\newtheorem{proposition}{Proposition}[section]
\newtheorem{lemma}{Lemma}[section]
\newtheorem{remark}{Remark}[section]

\newcommand{\R}{\mathbb{R}}
\newcommand{\N}{\mathbb{N}}
\newcommand{\F}{\mathcal{F}}

\renewcommand{\S}{\mathcal{S}}
\newcommand{\floor}[1]{\lfloor #1\rfloor}

\usepackage{enumerate}

\title[Optimal spectral differentiation]{Optimal stability of regularized spectral differentiation in Sobolev spaces}

\author{Teemu Tyni}
\address{Research Unit of Applied and Computational Mathematics,\\
University of Oulu, Finland}
\email{teemu.tyni@oulu.fi}

\begin{document}

\begin{abstract}
We study the problem of stable spectral differentiation of functions in Sobolev spaces from noisy data. We introduce a class of admissible Fourier multipliers under simple and directly verifiable conditions, and show that the corresponding regularized differentiation operators achieve minimax optimal stability rates. The results extend the previous $L^2$ based results to Sobolev spaces $H^{s,p}$, $1<p<\infty$. The analysis relies on multiplier estimates and applies to a wide class of multipliers, including Gaussian, spectral cutoff, and Tikhonov-type regularizations. Numerical examples demonstrate the behavior of several admissible spectral multipliers.
\end{abstract}

\maketitle

\section{Introduction}

In practical applications, one often seeks to compute derivatives of functions from noisy measurements. Direct differentiation is unstable and a common approach is to regularize differentiation in the Fourier domain.
To motivate the problem, consider a smooth enough function $f:\R\to\mathbb{C}$  and let
\[
\widehat f(\xi) = \frac{1}{\sqrt{2\pi}}\int_{-\infty}^\infty f(x)e^{-ix\cdot\xi}dx,
\]
be its Fourier transform. Then it is a straightforward application of integration by parts to verify that
\[
\widehat{f^{(k)}}(\xi) = (i\xi)^k \widehat f(\xi).
\]
However, in practice, one measures noisy data $f_\delta$ instead of $f$ and the possible high-frequency noise will blow up when multiplied by the factor $(i\xi)^k$. A natural approach is to dampen high frequencies with a cutoff multiplier $\chi_R$, which equals $1$ when $|\xi|\leq R$ and zero otherwise. This produces an approximation
\[
\widehat{f^{(k)}_{R,\delta}}(\xi) = (i\xi)^k \chi_R(\xi)\widehat f_\delta(\xi)
\]
of $\widehat{f^{(k)}}(\xi)$. Such a procedure was considered in \cite{qian_fourier_2006}, where it was proven that
\[
\Vert f^{(k)}-f^{(k)}_{R,\delta}\Vert_{L^2}\leq CE^{k/s}\delta^{1-k/s},
\]
when the function $f$ belongs to an $L^2$ based Sobolev space $H^{s,2}(\R)$ with $0\leq k < s$ and $\Vert f \Vert_{H^{s,2}}\leq E$, and the noisy measurements $f_\delta\in L^2(\R)$ satisfy $\Vert f - f_\delta\Vert_{L^2}<\delta$.
In the present work, we construct a general family of Fourier multipliers which produces the minimax optimal convergence rates
\begin{equation*}
\inf_{\substack{\mathcal A :L^p\to L^p\\\mathcal{A}(0)=0}}
\sup_{\substack{
    \|f\|_{H^{s,p}} \le E \\
    \|f - f_\delta\|_{L^p} \le \delta
}}
\| \partial^\alpha f - \mathcal A f_\delta \|_{L^p}
 \asymp E^{|\alpha|/s} \delta^{1-|\alpha|/s}
\end{equation*}
with respect to the noise level $\delta$ and the Sobolev norm $\Vert f \Vert_{H^{s,p}}\leq E$ for $1<p<\infty$ and $0\leq|\alpha|<s$.

There are many other strategies to tackle the numerical differentiation of noisy data.
A classical way to regularization of differentiation on a \emph{bounded} interval $[0,1]$ is to view it as an inverse problem. By the fundamental theorem of calculus $f(t) = f(0) + \int_0^t f'(s)\,ds$, so setting $x=f'$, $y=f(t)-f(0)$, and $Ax = \int_0^t x(s)ds$, one obtains a classical linear inverse problem of the form $Ax=y$, see~\cite{Cullum_numerical_differentiation,Chartrand}. However, this formulation does not extend directly to unbounded domains or higher dimensions.
Although regularized differentiation has been studied extensively, many existing analyses focus on one-dimensional settings, specific filter choices, Hilbert-space formulations, or particular applications. Spectral damping is considered in \cite{elden_wavelet_2000, qian_fourier_2006,yang_generalized_2014}, where simple truncation and Tikhonov mollification are used. Direct approaches considering differentiation as an inverse problem were considered for smooth data using Tikhonov methods in \cite{Wang2002,Cullum_numerical_differentiation} and for non-smooth functions using total variation regularization in \cite{Chartrand}, and using optimization methods in \cite{knowles_variational_diff,Hanke2001}. Moreover, some existing general spectral regularization theories such as \cite{Herdman2009} imply analogous conclusions in more abstract Hilbert space settings. Banach-space extensions based on source conditions and Bregman distance convergence rates have also been developed, for example, in \cite{EnglBook,Hofmann2007,SchusterBook}. These abstract frameworks apply to broad classes of inverse problems but generally do not provide directly verifiable conditions for particular classes of regularization operators. For an overview of different methods, we refer to \cite{knowles_methods_2014, smirnova} and the references therein. Recently, related regularization methods using spectral and Fourier multiplier techniques have also appeared in \cite{Rozendaal,Mathe,Benning_Burger_2018,Guastavino}.

The need for regularization in numerical differentiation stems from the fact that differentiation is an ill-posed problem in the sense of Hadamard~\cite{Hadamard}: small errors in the measurement are greatly amplified by differentiation. Along with the classical Hadamard definition, several further classifications of ill-posed inverse problems have been developed, including Nashed's distinction between ill-posedness of types I and II, and subsequent extensions to Banach spaces \cite{Nashed1987,Veselic}, providing additional context for regularization of differentiation. In the present work, we concentrate on explicit Fourier multiplier analysis and the stability estimates this approach provides.

The main contribution of this work is to identify simple verifiable conditions on Fourier multipliers that guarantee minimax optimal differentiation rates in the Sobolev spaces $H^{s,p}(\mathbb{R}^n)$. As a consequence, the resulting theory extends previously known one-dimensional $L^2$-Hilbert space based results to arbitrary dimensions, unbounded domains, and general (Banach) $L^p$-based Sobolev spaces, providing also a theoretically justified method for choosing the regularization parameter.
In contrast to abstract settings, the admissibility conditions here are stated directly in terms of Fourier multipliers and can be verified explicitly for commonly used filters.

In Section~\ref{sec: Stability} we establish the corresponding stability estimates, while Section~\ref{sec: Optimality} shows that these rates are minimax optimal. Section~\ref{sec: Examples} demonstrates that the framework includes sharp spectral truncation, smooth cutoff and Gaussian filters, generalized Tikhonov regularization, among other common strategies, finishing with brief numerical examples.

\section{Regularized spectral differentiation}\label{sec: Stability}

Let us introduce the notation for Fourier transform and Sobolev spaces. The Fourier transform is defined first on the Schwartz class $\S(\R^n)$ of smooth rapidly decaying functions by
\[
\widehat f(\xi) = \F(f)(\xi) \vcentcolon=
\frac{1}{(2\pi)^\frac{n}{2}}
\int_{\R^n} f(x) e^{-ix\cdot\xi} dx.
\]
With this normalization, the Fourier transform extends to a unitary operator on $L^2(\R^n)$ (Plancherel's theorem). Moreover, the Fourier transform  extends to a continuous automorphism on tempered distributions $\S'(\R^n)$ via duality~\cite[Ch.~7]{Hormander_book}. We denote by $\F^{-1}$ the inverse Fourier transform.  

The fractional Sobolev spaces $H^{s,p}(\R^n)$ for $s\in\R$ and $1\leq p\leq \infty$ are defined in the usual way as the Bessel potential spaces
\[
H^{s,p}(\R^n)\vcentcolon=
\{
f\in \S'(\R^n) \mid \F^{-1}( (1+|\xi|^2)^{s/2} \widehat f) \in L^p(\R^n)\}
\]
and they come equipped with the norm
\[
\Vert f \Vert_{H^{s,p}} \vcentcolon = \Vert \F^{-1}( (1+|\xi|^2)^{s/2} \widehat f) \Vert_{L^p} = \Vert (1-\Delta)^{s/2}f\Vert_{L^p}.
\]
For theory of Sobolev spaces, we refer to \cite{GrafakosClassical,  stein_introToFourier_1971,AdamsFournier}. 

Let $\alpha\in\N^n$ be a multi-index. Throughout, the symbol
\[
\partial^\alpha \vcentcolon= \frac{\partial^{|\alpha|}}{\partial x_1^{\alpha_1}\cdots \partial x_n^{\alpha_n}}
\]
denotes the usual partial derivatives in multi-index notation, where $|\alpha|=\alpha_1+\ldots+\alpha_n$. The symbol $C>0$ is used to denote a generic constant whose value can change from line to line, when it is not important to keep track of its precise value.

Let us start by defining an admissible symbol class for regularized spectral differentiation. We require two conditions. The first is a  Mikhlin condition, which guarantees that the multipliers map $L^p$ to itself. The second condition is an approximation property at low frequencies.

\begin{definition}[Admissible spectral multipliers]\label{def: admissible multiplier}
Let  $1< p<\infty$, $s>0$, and $\alpha\in \mathbb N^n$ be a multi-index with $0\leq |\alpha|<s$.
A family of Fourier multipliers $\{ m_{R,\alpha}(\xi) \}_{R>0}\subset L^\infty(\R^n)$ is called \emph{admissible of order $\alpha$ on $H^{s,p}$} if the following holds:
\begin{enumerate}[(a)]
    \item\label{def: Banach case} Case $p\neq 2$. For all $R>0$ the functions $m_{R,\alpha}\in C^\infty(\R^n\setminus\{0\})$ and there exists constants $c_\beta,c_\beta'>0$ independent of $R$ such that: 
\begin{equation}\label{def: Mikhlin condition}
    |\xi|^{|\alpha|} |\partial^\beta m_{R,\alpha}(\xi)| \leq c_\beta R^{|\alpha|}|\xi|^{-|\beta|}
\end{equation}
\begin{equation}\label{def: approximation condition}
    |\partial^\beta(1-m_{R,\alpha}(\xi))|\leq c_\beta' \min\{1,(|\xi|/R)^{s-|\alpha|}\}|\xi|^{-|\beta|}
\end{equation}
for all $\xi\in\R^n\setminus\{0\}$ and multi-indices $|\beta|\leq \floor{n/2}+1$.
\item\label{def: Hilbert case} Case $p=2$. It is enough that $m_{R,\alpha}\in L^\infty(\R^n)$ for all $R>0$ and the conditions \eqref{def: Mikhlin condition}-\eqref{def: approximation condition} hold with $\beta=0$.
\end{enumerate}
\end{definition}
Note that the assumptions in (\ref{def: Banach case}) of Definition~\ref{def: admissible multiplier} are strictly stronger than (\ref{def: Hilbert case}). The distinction between the two cases is due to the fact that for $p=2$ boundedness of Fourier multipliers follows directly from Plancherel's theorem, whereas for $p\neq2$ we will use the Mikhlin multiplier theorem.
\begin{definition}[Regularized differentiation operator]\label{def: Regularized differentiation}
Let $f\in \mathcal{S}'(\R^n)$ and $\alpha\in\N_0^n$. Let $m_{R,\alpha}$  be an admissible multiplier of order $\alpha$ on $H^{s,p}$ and define the \emph{regularized derivative} of order $\alpha$ as the Fourier multiplier operator
\begin{equation}\label{def: D_R}
    \widehat{D_R^{\alpha} f}(\xi) := (i\xi)^\alpha m_{R,\alpha}(\xi) \widehat f(\xi), \qquad \xi\in\R^n.    
\end{equation}

\end{definition}

Before discussing the main result, let us state a multiplier estimate.
\begin{lemma}\label{lemma: Mikhlin multipliers}
Let $s>0$ and let $\alpha\in\mathbb N^n$ satisfy $|\alpha|<s$.
Let $m_{R,\alpha}$ be admissible of order $\alpha$ on $H^{s,p}(\R^n)$ with $p\neq 2$. Define
\[
M_{R,\alpha}(\xi)
\vcentcolon=
(i\xi)^\alpha (1-m_{R,\alpha}(\xi))(1+|\xi|^2)^{-s/2}
\]
and
\[
n_{R,\alpha}(\xi)
\vcentcolon=
(i\xi)^\alpha m_{R,\alpha}(\xi).
\]
Then for every multi-index $\beta$ with  $|\beta|\leq \floor{n/2}+1$ there exist constants $C_\beta,C_\beta'>0$, independent of $R$, such that
\[
|\partial^\beta M_{R,\alpha}(\xi)|
\le
C_\beta R^{|\alpha|-s}|\xi|^{-|\beta|},
\]
and
\[
|\partial^\beta n_{R,\alpha}(\xi)|
\le
C_\beta' R^{|\alpha|}|\xi|^{-|\beta|},
\]
for all $\xi\in\R^n\setminus\{0\}$. 

In particular, the regularized differentiation operator \eqref{def: D_R} defines a bounded linear map $D_R^\alpha:L^p(\R^n)\to L^p(\R^n)$ with operator norm bounded by $\Vert D_R^{\alpha}\Vert_{L^p\to L^p}\leq C_\beta'R^{|\alpha|}$.
\end{lemma}

\begin{proof}
We use standard multiplier estimates.
Consider first
\[
M_{R,\alpha}(\xi)
=
(i\xi)^\alpha (1-m_{R,\alpha}(\xi))(1+|\xi|^2)^{-s/2}.
\]
By Leibniz' rule,
\[
\partial^\beta M_{R,\alpha}
=
\sum_{\beta_1+\beta_2+\beta_3=\beta}
C_{\beta_1,\beta_2,\beta_3}\,
\partial^{\beta_1}(i\xi)^\alpha\,
\partial^{\beta_2}(1-m_{R,\alpha})\,
\partial^{\beta_3}(1+|\xi|^2)^{-s/2}.
\]
Note that if $\beta_1 \nleq \alpha$ (component-wise), then $\partial^{\beta_1}(i\xi)^\alpha \equiv 0$, so we have the crude bound
\begin{equation}\label{eq: polynomials estimated}
|\partial^{\beta_1}(i\xi)^\alpha|
\leq
C_{\alpha,\beta_1}
|\xi|^{|\alpha|-|\beta_1|},\quad |\xi|\neq 0.
\end{equation}
Next, for every multi-index $\beta_3$,
\begin{equation}\label{eq: bracket estimated}
|\partial^{\beta_3}(1+|\xi|^2)^{-s/2}|
\le
C_{s,\beta_3}
(1+|\xi|)^{-s-|\beta_3|}.    
\end{equation}
Indeed, differentiation produces a finite sum of terms of the form
\[
P_{\beta_3}(\xi)(1+|\xi|^2)^{-s/2-|\beta_3|},
\]
where $P_{\beta_3}$ is a polynomial of degree at most $|\beta_3|$, which yields the stated bound \eqref{eq: bracket estimated}.
Finally, recall from \eqref{def: approximation condition} that
\[
|\partial^{\beta_2}(1-m_{R,\alpha})(\xi)|
\leq
c_{\beta_2}\min\{1,(|\xi|/R)^{s-|\alpha|}\}|\xi|^{-|\beta_2|}.
\]

Combining estimates \eqref{eq: polynomials estimated} and \eqref{eq: bracket estimated} with \eqref{def: approximation condition} yields two cases: 

\noindent
If $|\xi|\leq R$, then $\min\{1,(|\xi|/R)^{s-|\alpha|}\}= (|\xi|/R)^{s-|\alpha|}$ and 
\begin{align*}
&|\partial^{\beta_1}(i\xi)^\alpha|
\,|\partial^{\beta_2}(1-m_{R,\alpha}(\xi))|
\,|\partial^{\beta_3}(1+|\xi|^2)^{-s/2}|\\
&\quad\leq
C
|\xi|^{|\alpha|-|\beta_1|}
(|\xi|/R)^{s-|\alpha|}
|\xi|^{-|\beta_2|}(1+|\xi|)^{-s-|\beta_3|}\\
&\quad\leq
C R^{|\alpha|-s}
|\xi|^{|\alpha|-|\beta_1|}
|\xi|^{s-|\alpha|}
|\xi|^{-|\beta_2|}|\xi|^{-s-|\beta_3|}\\
&\quad\leq 
CR^{|\alpha|-s}
|\xi|^{-|\beta|}.
\end{align*}
\noindent
If $|\xi|> R$, then $\min\{1,(|\xi|/R)^{s-|\alpha|}\}\leq 1$ 
and $|\xi|^{|\alpha|-s}\leq R^{|\alpha|-s}$, so that
\begin{align*}
&|\partial^{\beta_1}(i\xi)^\alpha|
\,|\partial^{\beta_2}(1-m_{R,\alpha}(\xi))|
\,|\partial^{\beta_3}(1+|\xi|^2)^{-s/2}|\\
&\quad\leq
C
|\xi|^{|\alpha|-|\beta_1|}
|\xi|^{-|\beta_2|}|\xi|^{-s-|\beta_3|}\\
&\quad=
C|\xi|^{|\alpha|-s-|\beta|}\\
&\quad
\leq 
C R^{|\alpha|-s}|\xi|^{-|\beta|}.
\end{align*}
Summing over all combinations $\beta=\beta_1+\beta_2+\beta_3$
yields
\[
|\partial^\beta M_{R,\alpha}(\xi)|
\le
C_\beta
R^{|\alpha|-s}
|\xi|^{-|\beta|}.
\]

We now turn to
\[
n_{R,\alpha}(\xi)
=
(i\xi)^\alpha m_{R,\alpha}(\xi).
\]
Leibniz' rule gives
\[
\partial^\beta n_{R,\alpha}(\xi)
=
\sum_{\beta_1+\beta_2=\beta}
C_{\beta_1,\beta_2}\,
\partial^{\beta_1}(i\xi)^\alpha
\partial^{\beta_2}m_{R,\alpha}(\xi).
\]
Now it follows directly from \eqref{def: Mikhlin condition} that
\[
|\partial^{\beta_2} m_{R,\alpha}(\xi)| \leq
C R^{|\alpha|}|\xi|^{-|\alpha|-|\beta_2|}.
\]
Therefore
\begin{align*}
|\partial^{\beta_1}(i\xi)^\alpha
\partial^{\beta_2}m_{R,\alpha}|
\leq
C |\xi|^{|\alpha|-|\beta_1|}R^{|\alpha|}|\xi|^{-|\alpha|-|\beta_2|}
=
CR^{|\alpha|}|\xi|^{-|\beta|}.
\end{align*}
Summing again over all decompositions of $\beta=\beta_1+\beta_2$ gives the claim.

The regularized differentiation operator \eqref{def: D_R} arises from the multiplier $n_{R,\alpha}$, and by Mikhlin's multiplier theorem (see~\cite[Theorem~6.2.7]{GrafakosClassical} or \cite[Section~3.2, Th.~2]{salo_function_spaces}) is therefore a bounded linear operator $L^p(\R^n)\to L^p(\R^n)$ with operator norm $\Vert D_R^{\alpha} \Vert_{L^p\to L^p}\leq C R^{|\alpha|}$.
\end{proof}

Next we prove the main result of this paper, namely, the stability estimates for spectral differentiation in Sobolev spaces. In the $L^2$ setting, related stability estimates were obtained in joint work with Anttila and Harju~\cite{AnttilaHarjuTyni2025} (submitted) in the context of an inverse problem for nonlinear wave equations.
The proof of the following theorem is based on the multiplier estimates of Lemma~\ref{lemma: Mikhlin multipliers} and the application of the Mikhlin multiplier theorem, see~\cite[Theorem~6.2.7]{GrafakosClassical} or \cite[Theorem~8.10 and Corollary~8.11]{Duoandikoetxea}; see also the original works of Mikhlin~\cite{Mikhlin} and H\"ormander~\cite{Hormander}.  Heuristically, the idea is to split the difference $\partial^\alpha f - D_R^\alpha f_\delta$ into two parts: the approximation error $\partial^\alpha f - D_R^\alpha f$ between the exact and regularized derivatives, estimated by the approximation property \eqref{def: approximation condition}, and the noise term $D_R^\alpha(f-f_\delta)$, which is controlled by the $L^p$-multiplier properties of $m_{R,\alpha}$ given by \eqref{def: Mikhlin condition}.
\begin{theorem}[Stability of spectral differentiation in $H^{s,p}$]
\label{thm:stability}
Let  $1<p<\infty$, $s>0$, and $\alpha\in\N_0^n$ with $0\leq |\alpha|<s$. Let $E,\delta>0$ and assume $f\in H^{s,p}(\R^n)$ with $\Vert f \Vert_{H^{s,p}}\leq E$ and that noisy measurements satisfy $f_\delta\in L^p(\R^n)$ with
\[
 \Vert f-f_\delta \Vert_{L^p} \leq \delta.
\]
Let $m_{R,\alpha}$ be admissible of order $\alpha$ on $H^{s,p}(\R^n)$. Then
\[
\Vert\partial^\alpha f - D_R^{\alpha} f_\delta \Vert_{L^p} \leq C \Big( R^{|\alpha|-s} E + R^{|\alpha|} \delta \Big),
\]
where $C>0$ depends only on $\alpha,s,n$ (and multiplier constants). Choosing
\begin{equation}\label{eq: regularization parameter}
R = \Big( \frac{E}{\delta} \Big)^{1/s}    
\end{equation}
yields the bound
\[
\Vert\partial^\alpha f - D_R^{\alpha} f_\delta\Vert_{L^p} \leq C E^{|\alpha|/s} \delta^{1-|\alpha|/s}.
\]
\end{theorem}
\begin{proof}
We decompose the error into approximation and noise terms as follows
\[
\partial^\alpha f - D_R^{\alpha} f_\delta
=
(\partial^\alpha f - D_R^{\alpha} f)
+
D_R^{\alpha}(f-f_\delta).
\]

Taking the Fourier transform of the first term gives
\[
\F(\partial^\alpha f - D_R^{\alpha} f)(\xi)
=
(i\xi)^\alpha (1-m_{R,\alpha}(\xi)) \widehat f(\xi).
\]
Denote
\[
g := (I-\Delta)^{s/2} f,
\qquad
\widehat g(\xi) = (1+|\xi|^2)^{s/2} \widehat f(\xi).
\]
Since $f\in H^{s,p}(\R^n)$ we have $g\in L^p(\R^n)$ and
\[
\|g\|_{L^p} = \|f\|_{H^{s,p}}.
\]

Write
\[
\widehat{\partial^\alpha f - D_R^{\alpha} f}
=
M_{R,\alpha}(\xi)\widehat g(\xi),
\]
where
\[
M_{R,\alpha}(\xi)
:=
(i\xi)^\alpha (1-m_{R,\alpha}(\xi))(1+|\xi|^2)^{-s/2}.
\]
Therefore
\[
\partial^\alpha f - D_R^{\alpha} f
=
T_{M_{R,\alpha}} g,
\]
where $T_{M_{R,\alpha}}$ is the Fourier multiplier operator with symbol $M_{R,\alpha}$.

We prove that
\[
\Vert T_{M_{R,\alpha}}g \Vert_{L^p(\R^n)} \leq 
C R^{|\alpha|-s} \Vert f\Vert_{H^{s,p}(\R^n)}.
\]
Suppose first that $p=2$. From \eqref{def: approximation condition} under Definition~\ref{def: admissible multiplier}~(\ref{def: Hilbert case}) we have
\[
|1-m_{R,\alpha}(\xi)|
\leq
C\min\{1,(|\xi|/R)^{s-|\alpha|}\}
\]
which implies for $|\xi|\leq R$ that
\begin{align*}
|M_{R,\alpha}(\xi)|
&= 
|\xi|^{|\alpha|}|1-m_{R,\alpha}(\xi)|(1+|\xi|^2)^{-s/2}\\
&\leq
C|\xi|^{|\alpha|}(|\xi|/R)^{s-|\alpha|}|\xi|^{-s}
\leq CR^{|\alpha|-s}.
\end{align*}
Similarly, if $|\xi|>R$ then
\[
|M_{R,\alpha}(\xi)|
\leq C|\xi|^{|\alpha|}|\xi|^{-s}
\leq C R^{|\alpha|-s}.
\]
By Plancherel's theorem, we immediately have
\begin{align*}
\Vert T_{M_{R,\alpha}}g \Vert_{L^2(\R^n)} 
&= 
\Vert M_{R,\alpha} \widehat g \Vert_{L^2(\R^n)}\\
&\leq
 \Vert M_{R,\alpha}\Vert_{L^\infty(\R^n)}\Vert g \Vert_{L^2(\R^n)}\\
&\leq 
C R^{|\alpha|-s} \Vert f\Vert_{H^{s,2}(\R^n)}.
\end{align*}

The non-trivial case is when $p\neq 2$. Recall from Lemma~\ref{lemma: Mikhlin multipliers} that
\[
|\partial^\beta M_{R,\alpha}(\xi)|
\le
C_\beta R^{|\alpha|-s}|\xi|^{-|\beta|},
\qquad
|\beta|\leq \lfloor n/2\rfloor+1.
\]
Therefore, $M_{R,\alpha}$ satisfies the classical Mikhlin condition 
so that the Mikhlin multiplier theorem implies
\[
\|T_{M_{R,\alpha}}\|_{L^p\to L^p}
\le
C R^{|\alpha|-s}.
\]
Applying this to $g$ gives
\[
\|\partial^\alpha f - D_R^{\alpha} f\|_{L^p}
=
\|T_{M_{R,\alpha}} g\|_{L^p}
\le
C R^{|\alpha|-s}\|g\|_{L^p}.
\]
Since $\|g\|_{L^p}=\|f\|_{H^{s,p}}$ we obtain
\begin{equation}\label{eq: bound for approximation term}
\|\partial^\alpha f - D_R^{\alpha} f\|_{L^p}
\le
C R^{|\alpha|-s} E.    
\end{equation}

Consider next the term $D_R^\alpha f-D_R^\alpha f_\delta$.
The operator $D_R^\alpha$ is a Fourier multiplier operator with the symbol
\[
n_{R,\alpha}(\xi)=(i\xi)^\alpha m_{R,\alpha}(\xi).
\]
If $p=2$ it follows directly from \eqref{def: Mikhlin condition} under Definition~\ref{def: admissible multiplier}~(\ref{def: Hilbert case}) that
\begin{equation}\label{eq: L2 bound}
|(i\xi)^\alpha m_{R,\alpha}(\xi)|
\leq
|\xi|^{|\alpha|} |m_{R,\alpha}(\xi)|
\leq
CR^{|\alpha|}.    
\end{equation}
On the other hand, if $p\neq 2$, Lemma~\ref{lemma: Mikhlin multipliers} shows that
\begin{equation}\label{eq: Lp bound}
 |\partial^\beta ((i\xi)^\alpha m_{R,\alpha}(\xi))|
\le
C_\beta R^{|\alpha|}|\xi|^{-|\beta|},
\qquad |\beta|\leq \floor{n/2}+1.   
\end{equation}
Now, if $p=2$ Plancherel's theorem with \eqref{eq: L2 bound} and if $p\neq 2$ Mikhlin's multiplier theorem (or just Lemma~\ref{lemma: Mikhlin multipliers}) with \eqref{eq: Lp bound} show that
\[
\|D_R^\alpha\|_{L^p\to L^p}
\le
C R^{|\alpha|}
\]
implying for $1<p<\infty$ that
\begin{equation}\label{eq: bound for noise term}
\|D_R^\alpha(f-f_\delta)\|_{L^p}
\le
C R^{|\alpha|}\|f-f_\delta\|_{L^p}
\le
C R^{|\alpha|}\delta.
\end{equation}

Combining the bounds \eqref{eq: bound for approximation term} and \eqref{eq: bound for noise term} gives
\[
\|\partial^\alpha f - D_R^\alpha f_\delta\|_{L^p}
\le
C\big(
R^{|\alpha|-s}E
+
R^{|\alpha|}\delta
\big).
\]
Finally, the choice $R=(E/\delta)^{1/s}$ produces
\[
\|\partial^\alpha f - D_R^\alpha f_\delta\|_{L^p}
\le
C E^{|\alpha|/s}\delta^{1-|\alpha|/s}
\]
as claimed.
\end{proof}

\begin{remark}\label{rem: Hilbert is simpler}
As seen in the proof of Theorem~\ref{thm:stability}, in the $L^2$ based case the proofs are much shorter. The reason is that in $L^2$ based Sobolev spaces Mikhlin's multiplier theorem is replaced by the Plancherel theorem, which implies that any bounded measurable function is a Fourier multiplier from $L^2(\R^n)$ to $L^2(\R^n)$ with norm equal to the sup-norm of the function itself. Therefore, conditions \eqref{def: Mikhlin condition} and \eqref{def: approximation condition} are enough with $m_{R,\alpha}\in L^\infty(\R^n)$ and $\beta=0$, and it is not necessary to check the multiplier estimates for higher derivatives. 
\end{remark}

\begin{remark}
Although we do not pursue sharp constants, one can optimize the choice of regularization parameter slightly by minimizing the function $F(R)=R^{|\alpha|-s}E+R^{|\alpha|}\delta$. The minimum is found at
\[
R^*=
\left(\frac{s-|\alpha|}{|\alpha|}
\right)^{\frac{1}{s}}
\left(
\frac{E}{\delta}
\right)^{\frac{1}{s}}, \quad \text{if }|\alpha|>0,
\]
which is \eqref{eq: regularization parameter} up to a constant factor.
\end{remark}

\begin{remark}
Theorem~\ref{thm:stability} gives an \emph{a priori} way of choosing the regularization parameter $R=(E/\delta)^{1/s}$ provided that the quantities $E,\delta,$ and $s$ are known. In practice, one does not necessarily know all the values $E$ and $s$, or the noise level $\delta$, and the choice of the regularization parameter is done \emph{a posteriori} experimentally or based on heuristic or adaptive strategies such as the L-curve method or Morozov's discrepancy principle~\cite{hansen2010discreteinversebook,mueller2012linearandnonlinearinversebook}. 
\end{remark}


\section{Optimality of the stability estimate}\label{sec: Optimality}

The stability estimate of Theorem~\ref{thm:stability} is optimal, as it can be proven that no operator $\mathcal A: L^p(\R^n)\to L^p(\R^n)$ with the consistency condition $\mathcal{A}(0)=0$ can achieve better stability than that of Theorem~\ref{thm:stability}. This follows from
\begin{theorem}\label{thm: Optimality}
    Let $1<p<\infty$, $s>0$, and $\alpha\in\mathbb{N}^n$ be a multi-index such that $0 \leq |\alpha| < s$. 
    Then there exists a constant $c_0>0$ depending only on $s,\alpha,p,$ and $n$, such that the following holds.
    Given $E>0$ and $0<\delta<E$ there exists $f\in H^{s,p}(\R^n)$ satisfying $\Vert f \Vert_{H^{s,p}}\leq E$ and $\Vert f \Vert_{L^p}<\delta$ for which
    \[
    \Vert \partial^\alpha f \Vert_{L^p} \geq c_0 E^{|\alpha|/s}\delta^{1-|\alpha|/s}.
    \]
\end{theorem}

Before proving this theorem, let us show how Theorem~\ref{thm: Optimality} implies the optimality of the rates of Theorem~\ref{thm:stability}. 
We begin by selecting an arbitrary operator $\mathcal A: L^p(\R^n)\to L^p(\R^n)$ with $\mathcal{A}(0)=0$. Given $0<\delta<E$, it suffices to find $f\in H^{s,p}(\R^n)$ and $f_\delta\in L^p(\R^n)$, so that $\Vert f \Vert_{H^{s,p}}\leq E$ and $\Vert f-f_\delta\Vert_{L^p}<\delta$, yet
\[
    \Vert \partial^\alpha f-\mathcal A (f_\delta) \Vert_{L^p} \geq c_0 E^{|\alpha|/s}\delta^{1-|\alpha|/s}.
\]
Now, assuming Theorem~\ref{thm: Optimality}, let $f$ be the function provided by that theorem, and let $f_\delta\equiv0$. Then 
\[
\partial^\alpha f - \mathcal A (f_\delta) = \partial^\alpha f - \mathcal A (0) = \partial^\alpha f,
\]
and Theorem~\ref{thm: Optimality} gives the claimed bounds. Consequently, we have
\begin{equation*}
\inf_{\substack{\mathcal A :L^p\to L^p\\\mathcal{A}(0)=0}}
\sup_{\substack{
    \|f\|_{H^{s,p}} \le E \\
    \|f - f_\delta\|_{L^p} \le \delta
}}
\| \partial^\alpha f - \mathcal A(f_\delta) \|_{L^p}
 \asymp E^{|\alpha|/s} \delta^{1-|\alpha|/s}
\end{equation*}
in terms of the noise level $\delta$ and the \emph{a priori} bound $E$.
Therefore, the convergence rates of Theorem~\ref{thm:stability} are minimax optimal: The operator $D^\alpha_R$ satisfies this bound, and no operator $\mathcal A:L^p(\R^n)\to L^p(\R^n)$ with $\mathcal{A}(0)=0$ can achieve better uniform bounds. 

The proof of Theorem~\ref{thm: Optimality} requires the following estimate for dilation operators on Sobolev spaces.
\begin{proposition}\label{prop: dilation}
Let $\lambda\geq 1$ and $1< p<\infty$. Then the dilation operator 
\[
\sigma_\lambda: H^{s,p}(\R^n)\to H^{s,p}(\R^n),\quad
(\sigma_\lambda f)(x) = f(\lambda x),
\]
is a bounded linear operator with
\[
\Vert \sigma_\lambda f \Vert_{H^{s,p}}\leq C(s,p,n) \lambda^{s-n/p}\Vert f \Vert_{H^{s,p}},\quad \forall f\in H^{s,p}(\R^n).
\]
\end{proposition}
This result is proved by Triebel in \cite[Ch.~3.4.1, Proposition~2\,(ii)]{Triebel1} for the Triebel-Lizorkin spaces $F^s_{p,q}(\R^n)$, and follows for Bessel potential spaces $H^{s,p}(\R^n)$ from the identification $H^{s,p}(\R^n)=F^s_{p,2}(\R^n)$, when $1<p<\infty$.  We refer to \cite{hovemann2025,Schneider} for more discussion on dilation operators in function spaces.

\begin{proof}[Proof of Theorem~\ref{thm: Optimality}]
We will construct a constant $c_0>0$ independent of $\delta$ and $E$, and a function $f$ satisfying the conditions $\Vert f \Vert_{L^p}<\delta$, $\Vert f \Vert_{H^{s,p}}\leq E$, and $\Vert \partial^\alpha f\Vert_{L^p}\geq c_0 E^{|\alpha|/s}\delta^{1-|\alpha|/s}$. For this, let $\varphi \in C^{\infty}_c(\R^n)$ be a nonzero test function with $\Vert\varphi\Vert_{L^p} = 1$ and $\partial^\alpha\varphi \not\equiv 0$, and let $f(x) = a \varphi(\lambda x)$ for some $a > 0, \lambda \geq 1$ to be fixed later. Now a direct computation gives
\begin{equation}\label{eq: Lp norm of f}
    \Vert f \Vert_{L^p} 
    = \Vert a \varphi(\lambda x) \Vert_{L^p}
    = a \lambda^{-n/p} \Vert \varphi \Vert_{L^p}    
\end{equation}
and similarly
\begin{equation}\label{eq: partial derivative of f}
    \Vert \partial^\alpha f \Vert_{L^p}
    = 
    a \lambda^{|\alpha|-n/p} \Vert\partial^\alpha \varphi\Vert_{L^p}.    
\end{equation}
Finally, the norm of $f$ in $H^{s,p}(\R^n)$ is estimated by Proposition~\ref{prop: dilation} using the dilation operator $\sigma_\lambda$ as
\begin{equation}\label{eq: norm of f in Hsp}
\Vert f \Vert_{H^{s,p}}
=
a \Vert \sigma_\lambda \varphi \Vert_{H^{s,p}}
\leq
C_0 a \lambda^{s-n/p} \Vert \varphi \Vert_{H^{s,p}}.
\end{equation}

Next, let us select the parameters
\begin{equation}\label{eq: a and lambda}
\lambda = \left( \frac{E}{\delta} \right)^{1/s} \geq 1 \quad\text{and}\quad
 a = \frac12 \min\{ \frac{1}{C_0\Vert \varphi \Vert_{H^{s,p}}}, 1 \} E^\frac{n}{sp} \delta^{1-\frac{n}{sp}}.   
\end{equation}
Substituting these choices into \eqref{eq: Lp norm of f} and \eqref{eq: norm of f in Hsp} gives
\[
\Vert f \Vert_{L^p}= a\lambda^{-n/p} \leq \frac12 E^\frac{n}{sp} \delta^{1-\frac{n}{sp}} E^{-\frac{n}{ps}} \delta^{\frac{n}{ps}} = \frac{\delta}{2}
\]
and 
\[
\Vert f \Vert_{H^{s,p}}
\leq C_0 a\lambda^{s-n/p} \Vert \varphi \Vert_{H^{s,p}}
\leq 
\frac12 
E^\frac{n}{sp} \delta^{1-\frac{n}{sp}} E^{1-\frac{n}{ps}} \delta^{(\frac{n}{p}-s)/s} = \frac{E}{2}.
\]
Therefore, $\Vert f \Vert_{L^p}<\delta$ and $\Vert f \Vert_{H^{s,p}}\leq E$.
Finally, a direct substitution of the parameters \eqref{eq: a and lambda} into \eqref{eq: partial derivative of f} yields
\[
\Vert \partial^\alpha f\Vert_{L^p} 
\geq c_0 E^\frac{|\alpha|}{s}\delta^{1-\frac{|\alpha|}{s}}
\]
with $c_0 = \frac14 \min\{ \frac{1}{C_0\Vert \varphi \Vert_{H^{s,p}}}, 1 \}\Vert \partial^\alpha\varphi \Vert_{L^p}$, a constant independent of $\delta>0$, and $E>0$.
\end{proof}

\section{Examples of admissible multipliers}\label{sec: Examples}
\noindent\textbf{Truncation:} The multiplier 
\[
m_{R}(\xi)
=
\begin{cases}
1,& \text{if } |\xi|<R,\\
0,& \text{otherwise},
\end{cases}
\]
is an admissible multiplier in $H^{s,2}(\R^n)$, because in the Hilbert case $p=2$ admissibility of Definition~\ref{def: admissible multiplier}~(\ref{def: Hilbert case}) only requires that $m_{R}\in L^\infty(\R^n)$ and the conditions \eqref{def: Mikhlin condition} and \eqref{def: approximation condition} hold for $\beta=0$. If $p\neq 2$ the admissibility fails even in the one-dimensional case due to lack of smoothness required. While in one dimension this does not prevent $m_{R}$ from being an $L^p$ multiplier, in higher dimensions Fefferman's ball multiplier theorem~\cite{Fefferman} shows that $m_{R}$ is not an $L^p$-multiplier for any $1<p<\infty$, when $p\neq 2$ and $n\geq 2$. The Hilbert case is therefore special because discontinuous spectral multipliers can remain admissible.

\smallskip
\noindent\textbf{Smooth cutoff:} Let us next show that there exists admissible multipliers in $H^{s,p}(\R^n)$ for $1<p<\infty$. Let $\chi\in C_c^\infty(\R^n)$ be such that $0\leq \chi(\xi) \leq 1$ for all $\xi\in\R^n$, $\chi(\xi)=1$, when $|\xi|\leq 1$ and $\chi(\xi)=0$ when $|\xi|>2$, and define
\begin{equation}\label{eq: smooth cutoff}
m_{R}(\xi)=\chi(\xi/R).    
\end{equation}
This is an example of a smooth admissible multiplier of order $\alpha$ in any $H^{s,p}(\R^n)$, when $s>0$ and $0\leq |\alpha|<s$. Let us verify the conditions of Definition~\ref{def: admissible multiplier}~(\ref{def: Banach case}). 

In the region $|\xi|<R$ one has $m_R(\xi)\equiv 1$, so the condition \eqref{def: approximation condition} is trivially true. Condition \eqref{def: Mikhlin condition} reads $|\xi|^{|\alpha|} \partial^{\beta} 1 \leq CR^{|\alpha|}|\xi|^{-|\beta|}$, which is certainly true, because  if $|\beta|=0$, then one uses $|\xi|<R$ so that $|\xi|^{|\alpha|}\leq R^{|\alpha|}$, and for $|\beta|>0$ the left-hand side vanishes identically.

If $|\xi|>2R$ then $m_R(\xi)\equiv 0$ and the condition \eqref{def: Mikhlin condition} is trivial. The condition \eqref{def: approximation condition} is also true, as now $1-m_R\equiv 1=\min\{1,(|\xi|/R)^{s-|\alpha|}\}$, and derivatives of $1-m_R\equiv 1$ vanish identically.

It remains to check $R\leq |\xi| \leq 2R$. Observe that the functions $(\partial^{\beta}\chi)\in C_c^\infty(\R^n)$ are bounded by uniform constants $\Vert \partial^{\beta}\chi\Vert_{L^\infty}$. Using
\[
R^{|\alpha|}\leq |\xi|^{|\alpha|}\leq 2^{|\alpha|} R^{|\alpha|}
\quad\text{and}\quad
2^{-|\beta|} R^{-|\beta|}\leq |\xi|^{-|\beta|}\leq R^{-|\beta|}
\]
we get by the chain rule applied to the composite function $\chi(\xi/R)$ that
\begin{align*}
|\xi|^{|\alpha|} |\partial^{\beta} m_R(\xi)|
&=
|\xi|^{|\alpha|} |\partial^{\beta}\big(\chi(\xi/R)\big)|
=
|\xi|^{|\alpha|} R^{-|\beta|} |(\partial^{\beta}\chi)(\xi/R)|\\
&\leq 
2^{|\alpha|+|\beta|}\Vert \partial^{\beta}\chi\Vert_{L^\infty} R^{|\alpha|} |\xi|^{-|\beta|}   
\end{align*}
verifying \eqref{def: Mikhlin condition}. Similarly, if $|\beta|\geq 1$
\[
|\partial^{\beta} (1-m_R(\xi))|
=
R^{-|\beta|} |(\partial^{\beta}\chi)(\xi/R)|
\leq 2^{|\beta|}\Vert \partial^{\beta}\chi\Vert_{L^\infty} |\xi|^{-|\beta|}
\]
and if $|\beta|=0$
\[
|1-m_R(\xi)|
\leq 1 = \min\{1,(|\xi|/R)^{s-|\alpha|}\}
\]
verifying \eqref{def: approximation condition}. 

While the non-smooth truncation above fails to produce Fourier multipliers in higher dimensions when $p\neq 2$, the smooth cutoff is a Schwartz function, hence an $L^p$-multiplier. This construction does not require $\chi$ to be radial. More similar multipliers can be obtained as tensor products of lower dimensional smooth cutoff functions.

\smallskip
\noindent\textbf{Gaussian:} The multiplier
\[
m_{R,\alpha}(\xi)=e^{-(|\xi|/R)^2}
\]
is admissible when $p=2$ if $\max\{0,s-2\}\leq |\alpha|<s$.
Recall first from Definition~\ref{def: admissible multiplier}~(\ref{def: Hilbert case}) and Remark~\ref{rem: Hilbert is simpler} that when $p=2$ it suffices to take $\beta=0$. Then, noting that $t^{|\alpha|}e^{-t^2}$ is bounded for $t\geq 0$ we get
\[
|\xi|^{|\alpha|}e^{-(|\xi|/R)^2} = R^{|\alpha|}\left(\frac{|\xi|}{R}\right)^{|\alpha|}e^{-(|\xi|/R)^2} \leq C R^{|\alpha|},
\]
verifying \eqref{def: Mikhlin condition}.
Using the fact that $1-e^{-x^2}\leq x^2$, one checks that
\[
|1-e^{-(|\xi|/R)^2}| \leq \min\{1,(|\xi|/R)^2\},
\]
which satisfies the condition \eqref{def: approximation condition} if $0\leq s-|\alpha|\leq 2$. Whether this condition can be circumvented will not be pursued here.

\smallskip
\noindent\textbf{Tikhonov:} Let $n=1$, $p=2$, $\sigma>0$, and $k$ be a non-negative integer $<\sigma$. Denote the regularization parameter by $\lambda=R^{-2(\sigma+k)}$. Consider the  generalized Tikhonov minimization problem
\[
f_{\lambda,\delta}=\operatorname{arg\,min}_{f}\{
\Vert f - f_\delta\Vert_{L^2}^2 + \lambda \Vert f^{(k)} \Vert_{\dot H^{\sigma,2}}^2
\},
\]
where the homogeneous Sobolev spaces are defined as
\[
\dot H^{\sigma,2}(\R) \vcentcolon=\{ f \in L^2(\R) \mid |\xi|^\sigma\widehat f\in L^2(\R)\}
\]
with the norm $\Vert f \Vert_{\dot H^{\sigma,2}}\vcentcolon=\Vert |\xi|^\sigma \widehat f\Vert_{L^2}$.
This model with a penalty term in $H^{\sigma,2}(\R)$, instead of the homogeneous space $\dot H^{\sigma,2}(\R)$, was studied in \cite{yang_generalized_2014}. 
Due to Parseval's equality, the above minimization problem is equivalent to
\begin{align*}
\widehat f_{\lambda,\delta}
=
\operatorname{arg\,min}_{\widehat f}\int_{-\infty}^\infty\left(|\widehat f - \widehat f_\delta|^2
+
\lambda |\xi|^{2k}|\xi|^{2\sigma}|\widehat f|^2\right)d\xi,
\end{align*}
which (the integral defining a strictly convex functional) admits the unique solution
\begin{equation}\label{eq: tikhonov minimizer}
\widehat f_{\lambda,\delta}(\xi) = \frac{1}{1+\lambda|\xi|^{2(\sigma+k)}}\widehat f_\delta(\xi),    
\end{equation}
as in \cite[Theorem~2.11 and eq.~(2.15)]{Kirsch} and analogously to \cite[eq.~(11)]{yang_generalized_2014}. 
Thus, regularized differentiation of $f_\delta$ can be formulated as
\[
D_R^kf_\delta(x) 
= 
f^{(k)}_{R,\delta}(x) =\mathcal{F}^{-1}\left( \frac{(i\xi)^k}{1+R^{-2(\sigma+k)}|\xi|^{2(\sigma+k)}}\widehat f_\delta\right)(x),
\]
arising from the multiplier
\[
m_{R}(\xi)\vcentcolon= \frac{1}{1+R^{-2(\sigma+k)}|\xi|^{2(\sigma+k)}}.
\]
This is an admissible multiplier of order $\alpha$ in $H^{s,2}(\R)$ provided $0\leq\alpha < s$ and $s-|\alpha|\leq 2(\sigma+k)$, see Definition~\ref{def: admissible multiplier}~(\ref{def: Hilbert case}) and Remark~\ref{rem: Hilbert is simpler}.
This example shows that Tikhonov regularization also has an interpretation in the class of admissible multipliers. The higher-dimensional case is analogous, if $k$th derivatives are replaced by the operator $(-\Delta)^{k/2}$.

Interestingly, because the Tikhonov penalty is quadratic, the associated multiplier involves powers $2(\sigma+k)$ giving a lot of extra smoothness, even though the penalty term only involves derivatives of order $\sigma+k$.
This is visible already from the minimizer \eqref{eq: tikhonov minimizer}. Interpreting the multiplier $1/(1+\lambda|\xi|^{2(\sigma+k)})$ appearing in \eqref{eq: tikhonov minimizer} as the symbol of the elliptic pseudodifferential operator $(1+\lambda(-\Delta)^{\sigma+k})^{-1}$, one sees that the minimizer $f_{\lambda,\delta}$ gains $2(\sigma+k)$ derivatives in the Sobolev scale, whenever $f_\delta\in L^2(\R)$.

\begin{remark}\label{rem: characterization of multipliers is delicate}
We may conclude that the admissibility conditions of Definition~\ref{def: admissible multiplier} provide a calculus for constructing minimax optimal regularized differentiation schemes. 
The multiplier formulation is natural for spectral regularization, but a complete characterization of minimax optimal operators remains open. In the case $p=2$, bounded translation invariant operators are precisely Fourier multiplier operators with essentially bounded symbols~\cite[Ch.~1 Th.~3.16]{stein_introToFourier_1971}. For general $L^p$ spaces the situation is more delicate and practical symbol conditions such as those of Mikhlin provide a convenient sufficient criteria for boundedness.
\end{remark}

\subsection{Numerical examples}
 \begin{figure}[ht]
    \centering
    \includegraphics[width=0.45\linewidth]{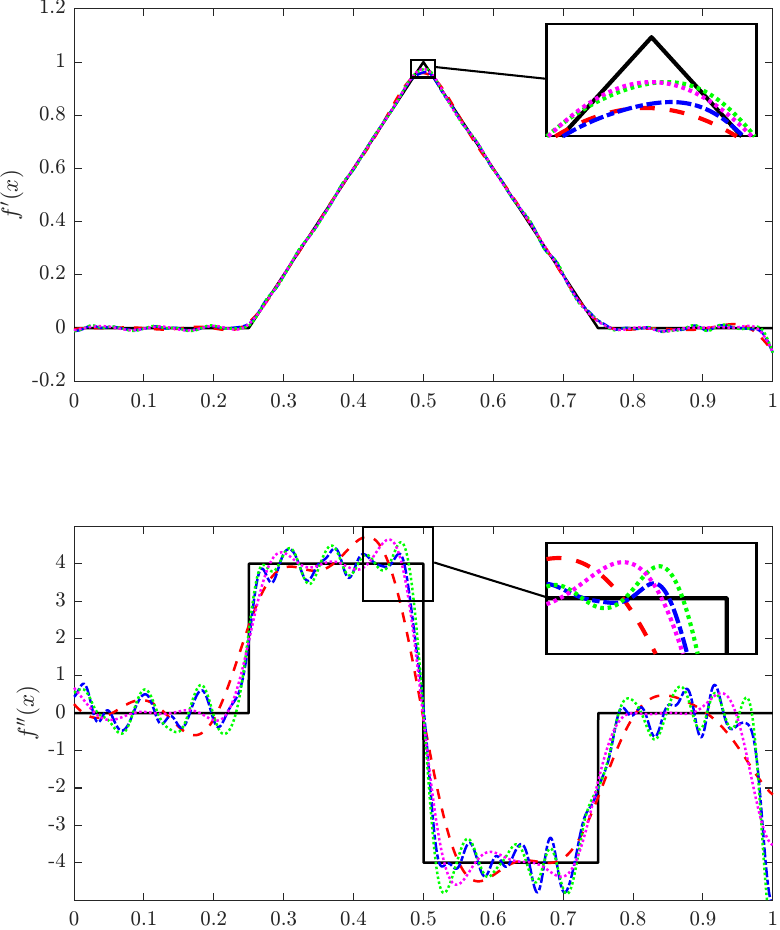}
    \hfill
    \includegraphics[width=0.45\linewidth]{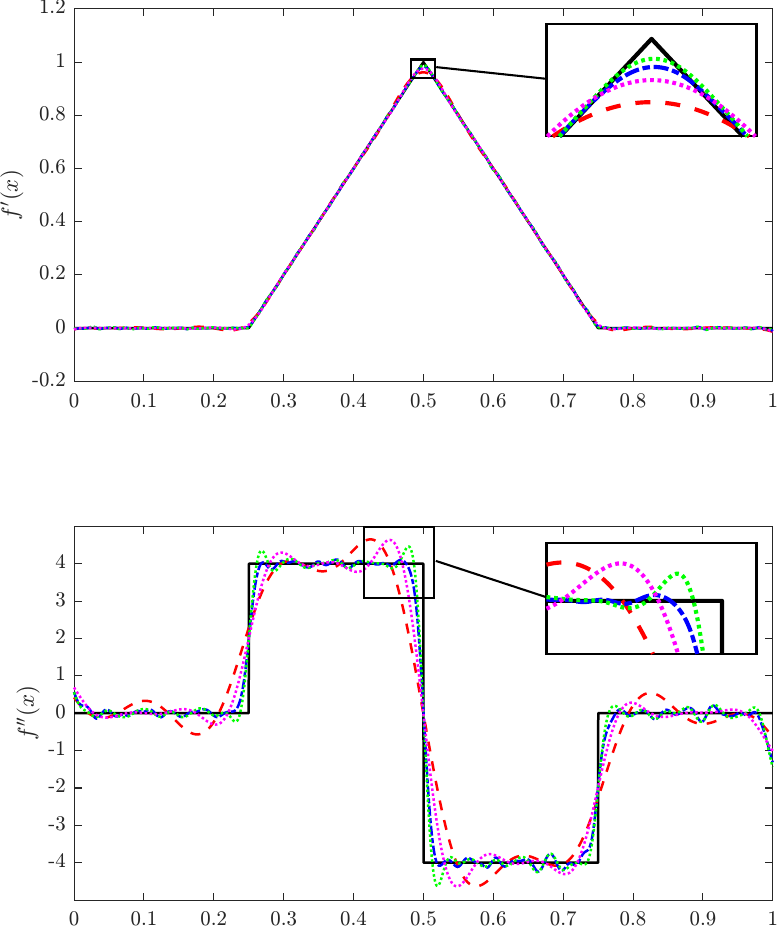}
    \caption{ 
    Comparison of regularized differentiation via truncation (red dashed), smooth cutoff (magenta dotted), Gaussian filtering (blue dot-dashed), and generalized Tikhonov (green dotted). The exact (weak) derivatives are shown in solid black. Sharp truncation causes oscillations, Gaussian smooths possible singularities, and the smooth cutoff and Tikhonov regularized solutions sit between the two.
    The examples are computed using Gaussian noise with standard deviations $\sigma=10^{-3}$ (left) and $\sigma=10^{-4}$ (right).
    }
    \label{fig:regu_diff_comparison}
\end{figure}

 \begin{figure}[ht]
    \centering
    \includegraphics[width=0.95\linewidth]{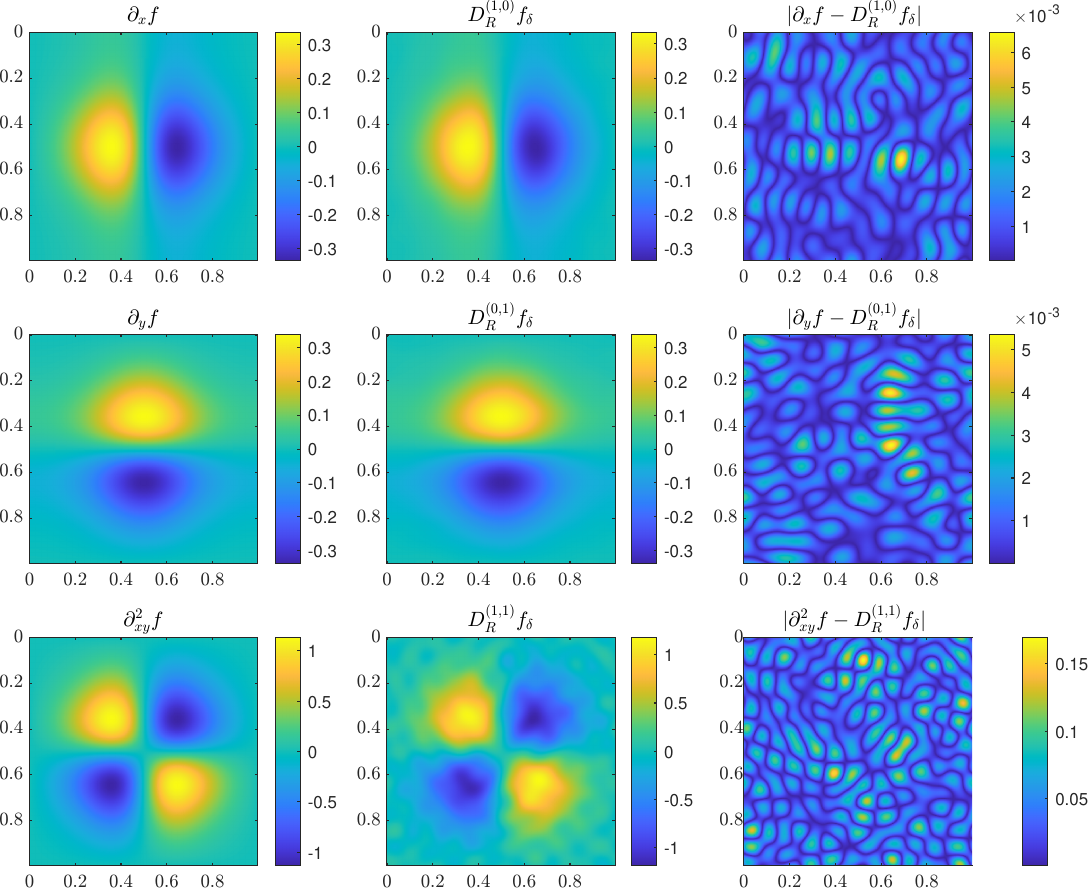}
    \caption{ 
    Spectral differentiation in 2D. Left column: True derivatives $\partial_x f(x,y)$, $\partial_y f(x,y)$, and $\partial_{xy}^2 f(x,y)$. Middle column: Results from spectral differentiation. Right column: Absolute difference $|\partial^\alpha f- D_R^\alpha f_\delta|$. Here $R=(E/\delta)^{(1/s)}\approx 7.7$ based on the numerically computed values $\delta\approx0.14$ and $E\approx 0.46$, and the choices $p=2$, $s=3$. Standard deviation of the noise is set to $\sigma=10^{-3}$.
    }
    \label{fig:2Dcase}
\end{figure}

\smallskip
Spectral differentiation can be efficiently performed via the discrete Fourier transform by employing the fast Fourier transform (FFT), using the discrete version of the formula
\[
D_R^\alpha f_\delta(x)=\F^{-1}((i\xi)^\alpha m_{R,\alpha}(\xi) \widehat f_\delta(\xi))(x).
\]
In the following examples, we use the theoretically justified choice $R= (E/\delta)^{1/s}$ with numerically computed values $E$ and $\delta$, which may be unrealistic in real-world applications.

\smallskip
\noindent\textbf{One-dimensional differentiation:}
We consider differentiation of a function $f\in H^{s,p}([0,1])$, where the Sobolev space $H^{s,p}([0,1])$ is defined through restriction as
\[
H^{s,p}([0,1]) \vcentcolon= \{ f\big|_{[0,1]} \mid f\in H^{s,p}(\R) \}.
\]
FFT requires periodic data, so for non-periodic data, we employ the smoothing spline extension proposed in \cite[Appendix]{elden_wavelet_2000} to enforce periodicity.

Let us briefly demonstrate the performance of the various multipliers. We select the function
\begin{equation}\label{eq: 1D example function}
f(x)=
\begin{cases}
    0, &  0\leq x < \frac14,\\
    2x^2 - x + \frac18, & \frac14 < x \leq \frac12,\\
    3x-2x^2-\frac78, & \frac12<x\leq\frac34,\\
    \frac14, & \frac34<x\leq 1,
\end{cases}  
\end{equation}
which is the same as in Example~5.2 of \cite{qian_fourier_2006}, and was also considered in \cite{AnttilaHarjuTyni2025}.
Now $f\in H^{s,p}([0,1])$ if and only if $s<2+1/p$. Indeed, $f'$ is a hat-function supported on $[\frac14,\frac34]$ and the weak derivative $f'' =4(\mathbf{1}_{[\frac14,\frac12]}-\mathbf{1}_{[\frac12,\frac34]})$, which shows that $f''\in H^{\varepsilon,p}([0,1])$ if and only if $\varepsilon<1/p$, see Lemma~\ref{lemma: characteristic in Sobolev}. 
This is an example of a function belonging to a genuinely fractional Sobolev space.
See also Figure~\ref{fig:regu_diff_comparison} for an illustration of the derivatives of $f$.
\begin{lemma}\label{lemma: characteristic in Sobolev}
Let $1<p<\infty$. Then the characteristic function $\mathbf{1}_{[0,1]}$ of the interval $[0,1]$ belongs to $H^{s,p}(\R)$ if and only if $s<1/p$.
\end{lemma}
\noindent The proof is postponed to Appendix~\ref{appendix}.

Since $f$ is non-periodic, we extend $f$ from $[0,1]$ to $[0,2]$ and then periodically over $\R$ using smoothing splines as in~\cite{elden_wavelet_2000}. The numerical examples are performed in \textsc{Matlab}, where the discrete version of $f_\delta$ is
\[
\mathbf{f}_\delta = \mathbf{f}+\sigma\, \mathtt{randn}(\mathtt{size}(\mathbf{f}))
\]
with $\mathbf{f}=(f(x_1),\ldots,f(x_n))$, $x_j=(j-1)/(N-1)$, $j=1,2,\ldots,N$ ($N=4097$). The command \texttt{randn}(\texttt{size}($\mathbf{f}$)) generates an array of the same size as $\mathbf{f}$ of normally distributed random numbers with $0$ mean and standard deviation~$1$. We use the standard deviations $\sigma=10^{-3}$ and $10^{-4}$ in the examples. The resulting solutions are shown in Figure~\ref{fig:regu_diff_comparison}.

\begin{table}[ht]
\centering
\begin{tabular}{|c|cc|cc|}
\hline
& \multicolumn{2}{c|}{1st derivative} & \multicolumn{2}{c|}{2nd derivative} \\
$p$ & Optimal rate & Observed rate & Optimal rate & Observed rate \\
\hline
$1.5$  & 0.625 & 0.5614 & 0.250 & 0.251 \\
$2$    & 0.600 & 0.550  & 0.200 & 0.202 \\
$10$   & 0.524 & 0.505  & 0.048 & 0.046 \\
\hline
\end{tabular}
\medskip
\caption{Convergence rates of numerical differentiation in $L^p$ spaces. The optimal rates for the example function \eqref{eq: 1D example function} are $1-|\alpha|/s$, where $s=2+1/p$, while the observed rates are computed as the slopes of a least-squares line fit to the measured error against the computed noise level $\delta$ in the corresponding $L^p$ space, see Figure~\ref{fig:errorplot}. The table demonstrates that convergence rates depend on the $L^p$ scale used.
}
\label{tab: convergence rates}
\end{table}

 \begin{figure}[ht]
    \centering
    \includegraphics[width=0.95\linewidth]{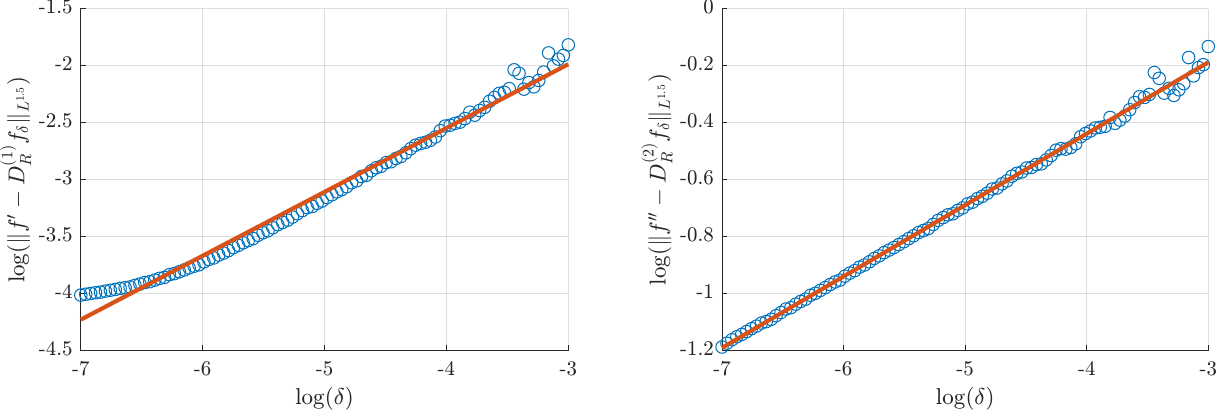}\\
    \includegraphics[width=0.95\linewidth]{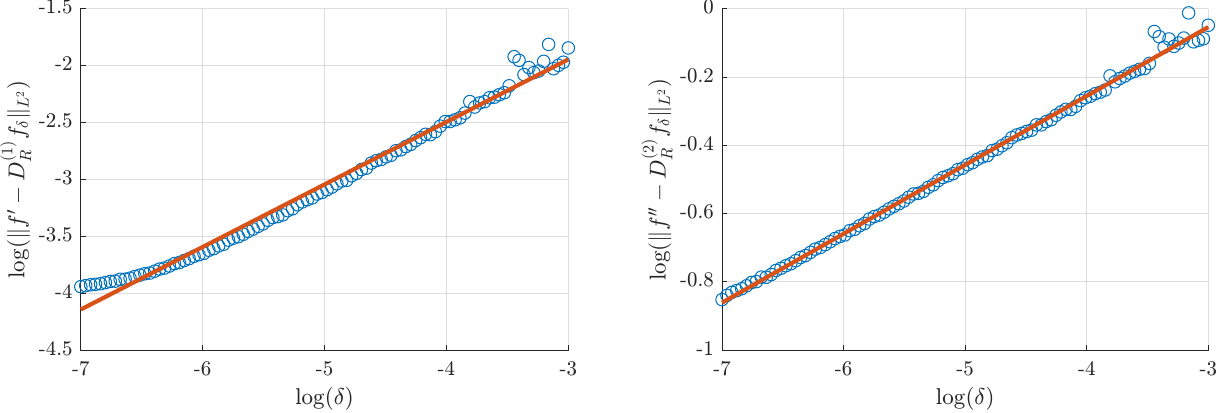}\\
    \includegraphics[width=0.95\linewidth]{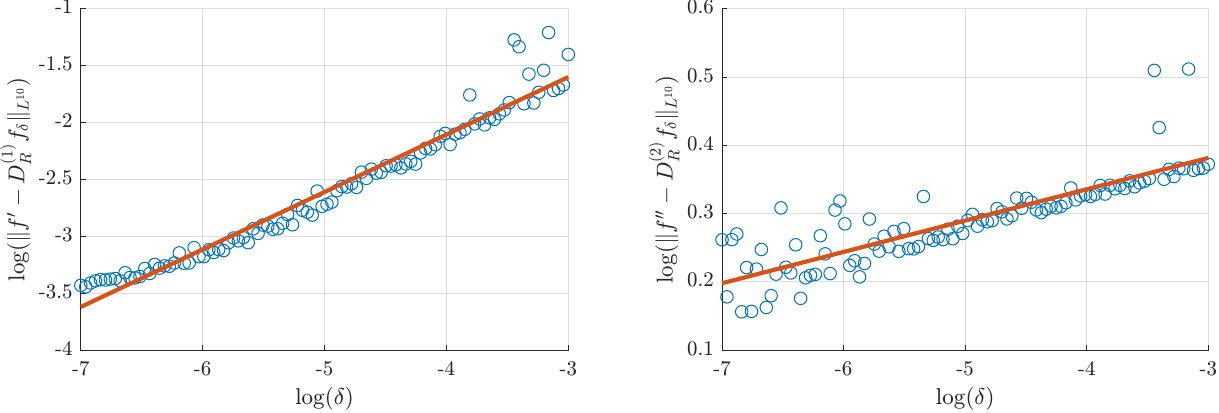}
    \caption{Convergence rates of the error with respect to the noise level in $\log\log$-scale. Here $R=R(\delta,s)$ is computed from $R=(E/\delta)^{1/s}$ with $s=2+1/p$ in three cases $p=1.5$ (top), $p=2$ (middle), $p=10$ (bottom). The multiplier is a smooth cutoff function.
    Blue circles present the measured error $\Vert f^{(k)}-D_R^{(k)}f_\delta\Vert_{L^p}$ ($k=1,2$ on the left and right, respectively) and the red line depicts the least-squares line fit to the error level taken as $\delta=\Vert f-f_\delta\Vert_{L^p}$.  
    The achieved convergence rates tabulated in Table~\ref{tab: convergence rates} are consistent with the theoretically predicted values.
     }
     \label{fig:errorplot}
\end{figure}

\smallskip
\noindent\textbf{Two-dimensional differentiation:}
Here we display spectral differentiation using a smooth cutoff multiplier \eqref{eq: smooth cutoff} in two-dimensions on the function
\[
f(x,y)=\frac{1}{10}e^{-2(\sin(\pi(x-0.5))^2 + \sin(\pi(y-0.5))^2)},\qquad (x,y)\in[0,1]^2.
\]
This function is a smooth $1$-periodic function in each variable. The implementation of two-dimensional derivatives is similar to the 1D case, and the results are shown in Figure~\ref{fig:2Dcase}.

\smallskip
\noindent\textbf{Convergence rates:}
Finally, in Figure~\ref{fig:errorplot} we test the convergence rates $\Vert f^{(k)}-D_R^{(k)}f_\delta\Vert_{L^p}\asymp E^{k/s}\delta^{1-k/s}$ for $k=1,2$ on the example function \eqref{eq: 1D example function} using the smooth cutoff function from \eqref{eq: smooth cutoff} as the multiplier, with a regularization parameter chosen as $R=(E/\delta)^{1/s}$ at various levels $\sigma\in[10^{-7},10^{-3}]$. In the manufactured example, we have $s<2+1/p$, and we include convergence tests for $p=1.5,2,$ and $10$. The measured convergence rates found in Table~\ref{tab: convergence rates} are consistent with the theoretical prediction.

\subsection*{Acknowledgements}

This work was supported by the Research Council of Finland (Flagship of Advanced Mathematics for Sensing, Imaging and Modelling grant 359186) and Emil Aaltonen foundation. I would like to thank the anonymous referees for their comments and suggestions, which have helped me to improve the manuscript.

\FloatBarrier
\appendix

\section{Proof of Lemma~\ref{lemma: characteristic in Sobolev}}\label{appendix}
Regularity properties of characteristic functions of domains have been investigated in generality in abstract Besov and Triebel–Lizorkin spaces, for example, in \cite{Sickel1,Sickel2} and the references therein. Although the result that $\mathbf{1}_{[0,1]}\in H^{s,p}(\R)$ if and only if $s<1/p$ for $1<p<\infty$ likely follows from these general results by embedding theorems, this one-dimensional result also follows by a short argument using Slobodeckij spaces and Sobolev embeddings, which we include for completeness.
\begin{proof}[Proof of Lemma~\ref{lemma: characteristic in Sobolev}]
We show first that $s<1/p$ is sufficient.
If $s\leq 0$ then $\mathbf{1}_{[0,1]}\in L^p(\R) = H^{0,p}(\R)\hookrightarrow H^{s,p}(\R)$ and there is nothing to prove, so we assume $s>0$. 
It is convenient to work with the Gagliardo seminorm defined for $0<s<1$ and $1\leq p<\infty$ by
\[
[f]_{W^{s,p}}:=\left( \int_{-\infty}^\infty\int_{-\infty}^\infty \frac{|f(x)-f(y)|^p}{|x-y|^{1+sp}}dxdy\right)^{1/p}.
\]
The Slobodeckij spaces $W^{s,p}(\R)$ are defined by finiteness of the norm
\[
\Vert f\Vert_{W^{s,p}}:= \Vert f \Vert_{L^p} + [f]_{W^{s,p}}.
\]
We show that $\mathbf{1}_{[0,1]}\in W^{s,p}(\R)$ if and only if $0<s<1/p$. For $1<p\leq 2$ the embedding $W^{s,p}(\R)\hookrightarrow H^{s,p}(\R)$ (see~\cite[Ch.~V, Th.~5(C)]{Stein}) then implies
$\mathbf{1}_{[0,1]}\in H^{s,p}(\R)$.
We compute the seminorm as
\[
[\mathbf{1}_{[0,1]}]_{W^{s,p}}^p = 2\int_0^1\int_{\R\setminus[0,1]} \frac{1}{|x-y|^{1+sp}}dxdy.
\]
Fix $y\in(0,1)$. Then the inner integral is
\begin{align*}
\int_{\R\setminus[0,1]} \frac{1}{|x-y|^{1+sp}}dx
&=
\int_{-\infty}^0 \frac{1}{|x-y|^{1+sp}}dx
+
\int_1^\infty \frac{1}{|x-y|^{1+sp}}dx
\\
&=
\frac{1}{sp}\left( y^{-sp} + (1-y)^{-sp}\right).
\end{align*}
Therefore
\begin{equation}\label{eq: Slobodeckij norm of indicator}
[\mathbf{1}_{[0,1]}]_{W^{s,p}}^p =\frac{2}{sp} \int_0^1\left( y^{-sp} + (1-y)^{-sp}\right)dy,    
\end{equation}
which converges if and only if $sp<1$. Thus $\mathbf{1}_{[0,1]}\in W^{s,p}(\R)$ if and only if $s<1/p$. Consequently, $\mathbf{1}_{[0,1]}\in H^{s,p}(\R)$ for $1<p\leq 2$ and $s<1/p$. 

Let next $2 \leq p < \infty$ and $s < 1/p$. Define
\[
t := s + \frac{1}{2} - \frac{1}{p}.
\]
Then $t < 1/2$ and since the case $p=2$ is already known, we have $\mathbf{1}_{[0,1]} \in H^{t,2}(\mathbb{R})$.
Moreover, because $t - \frac{1}{2} = s - \frac{1}{p}$ then by the Sobolev embedding
\[
H^{t,2}(\mathbb{R}) \hookrightarrow H^{s,p}(\mathbb{R})\quad\text{for }
2\leq p<\infty
\]
(see \cite[Th.~6.2.4(a)]{GrafakosModern})
we conclude that $\mathbf{1}_{[0,1]} \in H^{s,p}(\R)$ for $s<1/p$.

To show that $s<1/p$ is a necessary condition, note first that it is enough to show that $1/p\leq s<1$ is impossible, because if $s\geq 1$ then due to monotonicity of Bessel potential spaces $H^{s,p}(\R)\subset H^{s',p}(\R)$ for any $s'<1\leq s$.

Now, from \eqref{eq: Slobodeckij norm of indicator} we know that $\mathbf{1}_{[0,1]}\in W^{s,p}(\R)$ if and only if $s<1/p$. Thus the embedding $H^{s,p}(\R)\hookrightarrow W^{s,p}(\R)$ for $2\leq p<\infty$ and $0<s<1$ (see~~\cite[Ch.~V, Th.~5(A)]{Stein}) shows that $s<1/p$ is necessary, when $2\leq p<\infty$. 

Next, let $1<p<2$ and for the sake of a contradiction, assume that $\mathbf{1}_{[0,1]}\in H^{s,p}(\R)$ with $1/p\leq s<1$.
We know from earlier that $\mathbf{1}_{[0,1]}\in H^{t,q}(\R)$ for $q>2$ if and only if $t<1/q$. So, pick $q>2$ and define
\[
t:=s+\frac1q-\frac1p\geq \frac1p+\frac1q-\frac1p=\frac1q.
\]
The Sobolev embedding
\[
H^{s,p}(\R)\hookrightarrow H^{t,q}(\R)
\]
now shows that $\mathbf{1}_{[0,1]}\in H^{t,q}(\R)$, where $t\geq 1/q$, contradicting the earlier necessity of $t<1/q$ and concluding the proof.
\end{proof}

\bibliography{references} 
\bibliographystyle{abbrv}

\medskip
{\footnotesize
\noindent Teemu Tyni,
\textsc{Research unit of Applied and Computational mathematics, University of Oulu, Finland}\\
Email: \texttt{teemu.tyni@oulu.fi}
}

\end{document}